\documentclass[11pt]{amsart}

\usepackage{amscd}
\usepackage{xypic}  
\usepackage{amssymb}
\usepackage{amsthm}
\usepackage{epsfig}
\usepackage{url}
\usepackage{tikz}
\usepackage{hyperref}

\usepackage[all,cmtip]{xy}
\usepackage{amsmath}
\usepackage{color}

\newtheorem{thm}{Theorem}  

\newtheorem{lemma}[thm]{Lemma}
\newtheorem{prop}[thm]{Proposition}

\newtheorem*{claim A}{Claim A}
\newtheorem*{claim B}{Claim B}
\newtheorem*{claim C}{Claim C}

\newtheorem{corollary}[thm]{Corollary}

\theoremstyle{definition}

\def\min{\operatorname{min}}

\def\max{\operatorname{max}}

\def\c1{\operatorname{c_1}}
\def\c2{\operatorname{c_2}}

\def\gon{\operatorname{gon}}

\def\CC{{\mathbb C}}

\def\PP{{\mathbb P}}

\def\G{{\mathcal G}}

\def\M{{\mathcal M}}
\def\N{{\mathcal N}}
\def\O{{\mathcal O}}
\def\I{{\mathcal J}}

\def\T{{\mathcal T}}
\def\H{{\mathcal H}}

\def\e{\mathfrak{e}}

\def\c{\mathfrak{c}}

\def\x{\times}                   
\def\cong{\simeq}

\def\+{\oplus}               
\def\*{\otimes}     
\def\geq{\geqslant}
\def\leq{\leqslant}              

\def\Aut{\operatorname{Aut}}

\title{Pencils on general covers of an elliptic curve}

\begin{document}

\author[A.~L.~Knutsen]{Andreas Leopold Knutsen}
\address{A.~L.~Knutsen, Department of Mathematics, University of Bergen,
Postboks 7800,
5020 Bergen, Norway}
\email{andreas.knutsen@math.uib.no}

\author[M.~Lelli-Chesa]{Margherita Lelli-Chiesa}
\address{M.~Lelli-Chesa, Dipartimento di Matematica e Fisica, Universit{\`a} degli Studi Roma Tre, Largo San Leonardo Murialdo 1, 00146 Roma, Italy} 
\email{margherita.lellichiesa@uniroma3.it}

\maketitle

\begin{abstract}
  We completely describe the Brill-Noether theory of pencils on general primitive covers of elliptic curves of any degree.
  \end{abstract}

\section{Introduction}In recent years much focus has been directed toward the Brill-Noether theory of covers of $\PP^1$, see, e.g., \cite{Ap,Pf,JP,La,CJ}. Little is however known for covers of higher genus curves, even elliptic curves.
In this paper we settle the Brill-Noether theory of {\it pencils} of a general primitive cover of any fixed elliptic curve.

Fix any smooth irreducible elliptic curve $E$ over $\mathbb C$. A finite cover $\varphi:C \to E$ is called {\it primitive} if it does not factor through another finite cover $\varphi':C \to E'$ of smaller degree to an elliptic curve $E'$. The {\it Hurwitz scheme} $\H^{\tiny \mbox{prim}}_{g,k}(E)$ parametrizes primitive simply branched covers $\varphi:C \to E$ of degree $k$ from some smooth irreducible curve $C$ of genus $g\geq 2$. We recall that $\H^{\tiny \mbox{prim}}_{g,k}(E)$  is smooth and irreducible of dimension $2g-2$ \cite{GK,Buj}. We denote by $\M_{g,k}(E)$ the image  under the natural forgetful map of $\H^{\tiny \mbox{prim}}_{g,k}(E)$ in the moduli space $\M_g$ of genus $g$ curves. Since any cover can be composed with an automorphism of $E$, the locus $\M_{g,k}(E)$ is irreducible of dimension $2g-3$.It turns out (cf. Proposition \ref{intro})  that a general member of $\M_{g,k}(E)$ admits only one $k:1$ cover to $E$ up to the automorphisms of $E$ and, if $g>2$, no finite primitive maps to other elliptic curves. 

As the pullback of any $g^1_2$ on $E$ under a degree $k$ cover of it defines a $g^1_{2k}$ on the covering curve, any curve in $\M_{g,k}(E)$ has gonality at most $\min \left\{ 2k, \Big\lfloor \frac{g+3}{2} \Big\rfloor \right\}$ by standard Brill-Noether theory. We will prove the following fundamental result, which answers a natural question, yet hitherto unanswered:

\begin{thm} \label{thm:main}
  Let $\varphi:C\to E$ define a general point in the Hurwitz scheme $\H^{\tiny \mbox{prim}}_{g,k}(E)$ for $g\geq 2$ and $k\geq 2$. Let $G^1_d(C)^{nc}$ denote the closure of the locus of linear series in $G^1_d(C)$ that are base point free and not composed with $\varphi$. Then the following hold:
  \begin{itemize}
  \item[(i)] $G^1_d(C)^{nc}$ is nonempty if and only if either $k\geq 3$ and $\rho(g,1,d)\geq 0$, or $k=2$ and $d\geq g-1$;
  \item[(ii)] if nonempty, $G^1_d(C)^{nc}$ has pure dimension $\rho(g,1,d):=-g+2d-2$.
  \end{itemize}
    \end{thm}

     Part (ii), and thus the consequence that $C$ has gonality
$\min \left\{ 2k, \Big\lfloor \frac{g+3}{2} \Big\rfloor \right\}$,  was recently proved in \cite[Cor. 5.3]{Ke}.  Up to our knowledge, the only known part  of (i)  was the existence of a base point free $g^1_{g-1}$ in the case $k=2$ (cf. \cite{Pa} and references therein, where the result is obtained for any bielliptic curve and not only for a general one).
    
    The above theorem enables us  to fully  describe the Brill-Noether theory of pencils on a general curve $C$ in $\M_{g,k}(E)$, whose $k:1$ map onto $E$ we denote by $\varphi$.
    For every $2\leq n\leq d/k$, we denote by $G^1_{d,n}(C)$ the closure in $G^1_d(C)$ of the locus of linear series $\mathfrak g$ with $s:=d-nk$ base points $p_1,\ldots,p_s$ such that $\mathfrak g(-p_1-\cdots-p_s)$ is the pullback under $\varphi$ of a $g^1_n$ on $E$. It is straightforward that $G^1_{d,n}(C)$ is irreducible of dimension
    \begin{equation}
      \label{eq:star}
      \dim G^1_{d,n}(C)= \dim G^1_n(E)+d-nk=d-(k-2)n-3;
      \end{equation}
in particular, this does not depend on $n$ in the case $k=2$ . When $ \dim G^1_{d,n}(C)\geq\rho(g,1,d)$, that is, in the range $2\leq n\leq \frac{g-1-d}{k-2}$, Theorem \ref{thm:main} implies that $G^1_{d,n}(C)$ is an irreducible component of $G^1_d(C)$. More precisely, we obtain the following results:

  \begin{corollary}\label{ye}
    Fix $k\geq 3, g\geq 2$ and let $[C]\in \M_{g,k}(E)$ be general. Then
    $C$ has gonality
    \[ \gon C= \min \left\{ 2k, \Big\lfloor \frac{g+3}{2} \Big\rfloor \right\},\]
and for every $d \geq \gon C$  one has  \begin{equation}\label{dim}\dim G^1_d(C)=\max\{\rho(g,1,d),d-2k+1\}.\end{equation}
More precisely, the following hold:
    
    \begin{itemize}
  \item[(i)] if $\rho(g,1,d)\geq 0$, the irreducible components of $G^1_d(C)$ are the components of
$G^1_d(C)^{nc}$ and the loci
$G^1_{d,n}(C)$ for $2\leq n\leq \frac{g-1-d}{k-2}$;
\item [(ii)] if $\rho(g,1,d)< 0$, the irreducible components of $G^1_d(C)$ are the loci $G^1_{d,n}(C)$  for $2\leq n\leq \frac{d}{k}$.
   \end{itemize}
  \end{corollary}
In particular, when $k\geq 3$ the Brill-Noether varieties $ G^1_d(C)$ are not necessarily equidimensional. This fact highlights a big difference with the $k=2$ case, that we therefore state separately: 
  \begin{corollary}\label{ye2}
   Let $[C]\in \M_{g,2}(E)$ be general with $g\geq 2$. Then 
    $C$ has gonality \[ \gon C= \min \left\{ 4, \Big\lfloor \frac{g+3}{2} \Big\rfloor \right\},\] and 
    for every $d\geq \gon C$ the Brill-Noether variety $G^1_d(C)$ is equidimensional of dimension \begin{equation}\label{dim2}\dim G^1_d(C)=\max\{\rho(g,1,d),d-3\}.\end{equation}
  More precisely, the following hold:
    \begin{itemize}
    \item[(i)] if $d> g-1$, then  $G^1_d(C)$ coincides with $G^1_d(C)^{nc}$;
    \item[(ii)] if $d=g-1$,  the irreducible components of $G^1_d(C)$ are the components of $G^1_d(C)^{nc}$ and the loci $G^1_{d,n}(C)$  for $2\leq n\leq \frac{d}{2}$;
     \item[(iii)] if $d <g-1$,  the irreducible components of $G^1_d(C)$ are the loci $G^1_{d,n}(C)$ for every $2\leq n\leq \frac{d}{2}$.
    \end{itemize}
  \end{corollary}

  When $2k<\Big\lfloor \frac{g+3}{2} \Big\rfloor$, Corollaries \ref{ye} and \ref{ye2}  (or \cite[Cor. 5.3]{Ke})  imply that a general curve $[C]\in \M_{g,k}(E)$ does not satisfy Aprodu's {\it linear growth condition} \cite{Ap}. Green's Conjecture for these curves  is  nevertheless satisfied by
the recent result \cite[Thm. 5.4]{Ke} (the case $k=3$ had been done in in \cite{AF} the case $k=2$ is trivial, as  the curve $C$ is $4$-gonal and thus Green's conjecture reads like the ideal of the canonical curve $C\subset \PP^{g-1}$ being generated by quadrics).

\section{Proofs of the results}
First of all, we prove what anticipated at the beginning of the introduction:
\begin{prop}\label{intro}
Fix $g\geq 2,k\geq 2$ and any elliptic curve $E$. Then the following hold:
\begin{itemize}
\item[(i)] a general fiber of the forgetful map $\H^{\tiny \mbox{prim}}_{g,k}(E)\to \M_{g,k}(E)$ has dimension $1$ (corresponding to the automorphisms of $E$);
\item[(ii)] if $g\geq 3$, then a general $[C]\in \M_{g,k}(E)$ does not lie in any other $\M_{g,k'}(E')$ (where $k'\geq 2$ and $E'$ is an elliptic curve possibly conciding with $E$ if  $k\neq k'$).
\end{itemize}
\end{prop}
\begin{proof}
 Let $[\varphi:C \to E] \in \H^{\tiny \mbox{prim}}_{g,k}(E)$ be general and, by contradiction, assume the existence of another primitive simply branched map $\varphi':C\to E'$ of degree $k'\geq 2$ to an elliptic curve $E'$ (where $\varphi$ and $\varphi'$ are required to be distinct up to the automorphisms of $E$ in the case $E'=E$ and $k'=k$). We stress that both $\varphi$ and $\varphi'$ are primitive and so they are not composed one with the other. The pair $(\varphi,\varphi')$ then defines a map from $C$ to the abelian surface $E\times E'$, that can be factored as $C\stackrel{\alpha}{\longrightarrow} D\stackrel{\psi}{\longrightarrow}E\times E'$, where $D$ is a curve of genus $h\leq g$ (with $D=C$ when $h=g$), $\alpha$ is a finite map of degree $1\leq n<\min\{k,k'\}$ dividing both $k$ and $k'$, and the map $\psi$ is birational. In particular, the image $\psi (D)$ is a curve of numerical class $lE+l'E'$ with $l=k/n\geq 2$ and $l'=k'/n\geq 2$. By \cite[Prop. 4.16]{DS}, integral genus $h$ curves in $E\times E'$ of this numerical class move in dimension $h$. Modding out by the automorphisms of $E\times E'$, we obtain a family of dimension $\leq h-2$ of curves $D$ of genus $h\geq 2$ admitting a birational map to a curve of class $lE+l'E'$ in $E\times E'$. Letting $E'$ vary in moduli and denoting by $\H_{g,n}(D)$ the Hurwitz scheme of genus $g$ and degree $n$ simply branched covers of a curve $D$ as above, one computes that curves $C$ admitting a pair $(\varphi,\varphi')$ as described move in dimension at most
$$1+h-2+\dim \H^{\tiny \mbox{prim}}_{g,n}(D)=h-1+2g-2-n(2h-2),$$
which is $<2g-3=\dim\M_{g,k}(E)$ unless $n=1$ and $h=g=2$. 
\end{proof}

To prove our main results, we now consider the surface

$T:=E \x \PP^1$ and let $R$ and $E$ denote the classes of the obvious sections of
$T$, so that $K_T \equiv -2E$, 

Let $V_g(k,d)$ denote the Severi variety of irreducible nodal curves of genus $g$ on $T$ of class $kE+dR$. Similarly, we denote by $V^g(k,d)$ the equigeneric locus of integral curves of genus $g$ on $T$ of class $kE+dR$, so that one has an obvious inclusion $V_g(k,d)\subset V^g(k,d)$. Let $\M_g(k,d)$ be the coarse moduli space of genus $g$ stable maps with image of class $kE+dR$, and denote by $\M_g(k,d)^{\mathrm{sm}}$ the closure of the locus of maps in $\M_g(k,d)$ that are smoothable, that is, can be deformed to a map from a nonsingular curve, birational onto its image (cf. \cite{Va}). By sending a stable map $f:C\to T$ to its image $f(C)=\overline{C}$, one obtains a dominant morphism from the semi-normalization of $\M_g(k,d)^{\mathrm{sm}}$ to $\overline{V^g(k,d)}$ \cite[I.6]{Ko1}. The following result is of independent interest (part (ii) was aready proved in \cite[Prop. 5.2]{Ke}):
\begin{prop}\label{facile}
For every $1 \leq g \leq d(k-1)+1$ the following hold:
\begin{itemize}
\item[(i)] the variety $V_g(k,d)$ is nonempty, smooth of dimension $2d+g-1$ and dense in the equigeneric locus $V^g(k,d)$;
\item[(ii)] $\M_g(k,d)^{sm}$ is generically reduced of dimension $2d+g-1$.
\end{itemize}
\end{prop}
\begin{proof}
 Since $-K_T \cdot (kE+dR) =2d \geq 4$, it is well-known that $\M_g(k,d)^{sm}$  and $V^g(k,d)$ are generically reduced (as a general stable map in any irreducible component is unramified) and equidimensional of dimension $2d+g-1$ as soon as they are nonempty, and that $V_g(k,d)$ is dense in
 $V^g(k,d)$, cf., e.g., \cite[proof of Prop. 2.2]{CH}. For $[\overline{C}]\in V_g(k,d)$, we denote by $N$ the scheme of nodes of $\overline{C}$ and by $\nu:C\to  \overline{C}\subset T$ the composition of the normalization map with the inclusion of $\overline{C}$ in $T$. Since the tangent space of $V_g(k,d)$ at the point $[\overline{C}]$ can be identified with $H^0(T,\mathcal O_T(\overline{C})\otimes \I_N)$ and this is isomorphic to the $2d+g-1$-dimensional space $H^0( C, \N_\nu)=H^0(C, \omega_{C}\otimes \nu^*(-K_T))$ \cite[\S 3.1]{DS}, we conclude that $V_g(k,d)$ is smooth. As a consequence, to prove nonemptiness we first note that $\M_1(k,d)^{sm}$ is nonempty
  (just pick any smooth elliptic curve admitting a degree $k$ map onto $E$, together with a general $g^1_d$ on it). It follows that also $V_1(k,d)$ is nonempty and its smoothness implies that the nodes of the curves can be smoothed independently; as a consequence, $V_g(k,d)$ is nonempty for all $1 \leq g \leq p_a(kE+dR)=d(k-1)+1$.
\end{proof}

In the sequel we will make use of the following vanishing:

\begin{lemma}\label{claimA}
Let $N$  be the scheme of the $\delta:=d(k-1)+1-g$ nodes of a curve $\overline{C}$ in $V_g(k,d)$. Then
$$
h^0(T,\mathcal O_T(\overline{C}+K_T)\otimes \I_N)=g-1.$$
\end{lemma}
\begin{proof}
 Let $\mu_N:\widetilde{T}_N \to T$ be the blow up at $N$, with total exceptional divisor $\e$, and let $C$ be the normalization of $\overline{C}$, also coinciding with the strict transform of $\overline{C}$ under $\mu_N$. Since $C^2=\overline{C}^2-\delta=2kd-\left(d(k-1)+1-g\right)= kd+d+g-1>0$,  we have that $C$ is big and nef, so that
    \[ h^0(\O_{\widetilde{T}_N}(C+K_{\widetilde{T}_N}))=\chi(\O_{\widetilde{T}_N}(C+K_{\widetilde{T}_N}))=\frac{1}{2}C\cdot\left(C+K_{\widetilde{T}_N}\right)=g-1.\]
 But $C+K_{\widetilde{T}_N} \sim \mu_N^*\overline{C}-2\e+\mu_N^*K_T+\e=\mu_N^*(\overline{C}+K_T)-\e$, and the result follows.  
\end{proof}

The following result is used to specialize the nodes of curves in $V_g(k,d)$ to the nodes of some reducible curves.  
 \begin{lemma}\label{claimB}
Fix $g \geq 5$, $k\geq 3$, $d\geq (g-1)/2$ and two integers $l_1,l_2$ such that $l_1+l_2=2d+1-g$ and $0\leq l_i\leq d-2$ for $i=1,2$. Let $$f:\widetilde X=\widetilde D\cup_{Z_1} E_1\cup_{Z_2} E_2\to X=D\cup E_1\cup E_2\subset T$$ be a genus $g$ stable map such that $\widetilde D$ is a smooth elliptic curve of genus $1$ with image $D\equiv (k-2)E+dR$,  the curves $E_1,E_2$ are mapped isomorphically to elements of $|E|$, and $\widetilde D$ and $E_i$ intersect transversally at $d-l_i$ points denoted by $Z_i$ for $i=1,2$. Then $[f]\in \M_g(k,d)^{sm}$ and $D+E_1+E_2$ lies in the closure of $V_g(k,d)$.   \end{lemma}
\begin{proof}
By Proposition \ref{facile}, it is enough to show that $f$ is smoothable. We denote by $f_{D}:=f|_{\widetilde D}:\widetilde D\to D\subset T$, $f_{E_i}:E_i\to E_i\subset T$ for $i=1,2$ and $f_1:=f|_{\widetilde D\cup_{Z_1} E_1}:\widetilde D\cup_{Z_1} E_1\to D\cup E_1\subset T$ the restrictions of $f$. It is enough to prove that $[f]\in\M_g(k,d)^{sm}$ when $[f_D]$ is general  in any irreducible component of $\M_1(k-2,d)^{sm}$.
The map $f_1$ defines a point of $\M_{1+d-l_1}(k-1,d)$, whose expected dimension is $3d-l_1$. As soon as $l_1\leq d-2$, one has $3d-l_1>\dim \left(\M_1(k-2,d)^{sm}\times \M_1(1,0)\right)=2d+1$, where the equality follows from Proposition \ref{facile}; as a consequence, the map $f_1$ is smoothable, that is, it lies in a component $\widehat\M$ of $\M_{1+d-l_1}(k-1,d)^{sm}$. Take now a genus $g$-stable map $h: \widetilde Y\cup_ZE\to Y\cup E\subset T$ such that $h_Y:=h|_{\widetilde Y}$ corresponds to a general point of $\widehat\M$, the restriction $h_E:=h|_{E}$ defines a point of $\M_1(1,0)$ and $Z:=Y\cap E$ consists of $d-l_2$ nodes. Then $h$ defines a point of $\M_g(k,d)$, which has expected dimension equal to $2d+g-1$. Since $2d+g-1>\dim\left(\widehat\M\times \M_1(1,0)\right)=3d-l_1+1$ for $l_2\leq d-2$, we conclude that the map $h$ is smoothable and thus the same holds for $f$.\end{proof}

\begin{prop} \label{prop:severi}
For every $2 \leq g \leq d(k-1)+1$ and $d \geq 2k$, the forgetful map $\psi_{g,k,d}: \M_g(k,d)^{sm} \dashrightarrow \M_g$ dominates $\M_{g,k}(E)$ when $\rho(g,1,d) \geq 0$. 
\end{prop}

\begin{proof}
Let $f:C\to T$ define a general point in some component of $\M_g(k,d)^{sm}$ so that the image $\overline{C}:=f(C)$ is a curve with $\delta:=d(k-1)+1-g$ nodes, that we denote by  $N$. As in the proof of Lemma \ref{claimA}, let $\mu_N:\widetilde{T}_N \to T$ denote the blow up at $N$ and by $\e$ its exceptional divisor. Since $A:=f^*(\O_T(E))\in G^1_d(C)$ moves in dimension $\geq \rho(g,1,d)$, the fiber $\psi_{g,k,d}^{-1}([C])$ has dimension $$\geq \rho(g,1,d)+\dim \Aut(T)=\rho(g,1,d)+4= \dim \M_g(k,d)^{sm}-\dim \M_{g,k}(E).$$ We want to show that equality holds. 
  
  Let $\N_{f}$ denote the normal sheaf of $f$, fitting into
  \[
    \xymatrix{
0 \ar[r] & \T_{C} \ar[r] & f^*\T_T \ar[r] & \N_{f} \ar[r] & 0.
    }
    \]
The coboundary map
    \[ H^0(\N_{f}) \longrightarrow H^1(\T_{C})\]
is the differential
\[ d\psi_{g,k,d}: T_{[f]}\M_g(k,d)^{sm} \longrightarrow T_{[C]}\M_g \]
of the forgetful map $\psi_{g,k,d}$ at $[f]$.
Its kernel is isomorphic to
      \[ H^0(f^*\T_T) \cong H^0(f^*(\O_T \+ \O_T(2E))) \cong H^0(\O_{C}) \+ H^0(\O_{C}(2A)),\]
      which has dimension
      \begin{eqnarray*}
        1 +h^0(\O_{C}(2A))  & = & 1+(2d+1-g+h^0(\omega_{C}-2A))= \rho(g,1,d)+4+h^0(\omega_{C}-2A).
          \end{eqnarray*}
      To finish the proof of the proposition, it will therefore be enough to prove that when $\rho(g,1,d) \geq 0$ the vanishing
      \begin{equation}
        \label{eq:o-2a}
        h^0(\omega_{C}-2A)=0
      \end{equation}
      holds for a general $(C,A)$ where $C$ is the domain of a general $[f]$ in some component of $\M_g(k,d)^{sm}$ and $A=f^*(\O_T(E))$. The vanishing holds for degree reasons if $d \geq g$, so we may assume that $4 \leq 2k \leq d < g$. The case $k=2$ thus yields $\delta=d+1-g=0$, $d=g-1$ and $\rho(g,1,d)=g-4=d-3\geq 1$. The vanishing \eqref{eq:o-2a} thus holds as soon as $A$ is not a theta characteristic on $C$ and we may assume this is not the case as $A$ moves in a positive dimensional family. From now on we will therefore assume $6 \leq 2k \leq d < g$.

      From the short exact sequence
      \[
        \xymatrix{
0 \ar[r] & \O_{\widetilde{T}_N}(K_{\widetilde{T}_N}-2\mu_N^*E) \ar[r]^{\hspace{-0.3cm} s_N} & \O_{\widetilde{T}_N}(K_{\widetilde{T}_N}+C-2\mu_N^*E) \ar[r] & \omega_{C}-2A \ar[r] & 0 
        }
        \]
        and the facts that $K_{\widetilde{T}_N}\sim\mu_N^*K_T+\e$ and $C\sim\mu_N^*\overline{C}-2\e$, we see that \eqref{eq:o-2a} is equivalent to the two conditions
        \begin{equation}
          \label{eq:o-2a1}
          h^0(T,\mathcal O_T(\overline{C}-2E+K_T)\otimes\I_N)=0
        \end{equation}
        and
        \begin{equation}
          \label{eq:o-2a2}
   H^1(s_N): H^1(T,\mathcal O_T(-4E)) \longrightarrow H^1(T,\mathcal O_T(\overline{C}-2E+K_T)\otimes\I_N) \;\; \mbox{is injective}
        \end{equation}
   
   Let $V'$ be a component of $\overline{V^{d}(k,d)}$ containing a reducible curve $D+E_1+E_2$ as in Lemma \ref{claimB}, and let $\widehat{V}$ be a component of $\overline{V^g(k,d)}$ containing $V'$. We will prove the vanishings  \eqref{eq:o-2a1} and  \eqref{eq:o-2a2} for the scheme of nodes $N$ of a general curve $[\overline{C}]\in \widehat{V}$.

To prove \eqref{eq:o-2a1},
we may specialize $N$ to a scheme of nodes of a reducible curve $D+E_1+E_2$ (as in Lemma \ref{claimB}) containing all the nodes $N_D$ of $D$ plus $l_1\leq d-2$ nodes of $D \cap E_1$ and $l_2\leq d-2$ nodes of $D \cap E_2$ with $l_1+l_2=2d+1-g$. By Lemma \ref{claimA} applied to the curve $[D]\in V_1(k-2,d)$, we obtain the vanishing
$h^0(\mathcal O_T(\overline{C}-2E+K_T)\otimes \I_{N_D})=h^0(\mathcal O_T(D+K_T)\otimes \I_{N_D})=0$, so that the vanishing \eqref{eq:o-2a1} also holds .

We next specialize $N$ to a length-$\delta$ subscheme $N'$ of the scheme of nodes $N^+$ of a general curve $[\overline{C}^+]\in V'$, whose normalization is denoted by $C^+$. Since the genus of $\overline{C}^+$ is $d$, the vanishing $h^0(\omega_{C^+}-2A)=0$ holds for degree reasons, whence also 
\begin{equation} \label{eq:o-2a1+}
           h^0(T,\mathcal O_T(\overline{C}-2E+K_T)\otimes\I_{N^+})=0
        \end{equation}
        and
      \begin{equation}
          \label{eq:o-2a2+}
   H^1(s_{N^+}):H^1(T,\mathcal O_T(-4E)) \longrightarrow H^1(T,\mathcal O_T(\overline{C}-2E+K_T)\otimes\I_{N^+}) \;\; \mbox{is injective}.
        \end{equation}
        Let $W \subset N^+$ be any subset of $g-d$ points and set $N':=N^+ \setminus W$. From the short exact sequence
        \[
          \xymatrix{
            0 \ar[r] & \I_{N^+} \ar[r] & \I_{N'} \ar[r] & \O_W \ar[r] & 0
           }
          \]
          we obtain the commutative diagram with exact rows:
          \[
          \xymatrix{
            & H^1(\mathcal O_T(-4E)) \ar@{=}[r] \ar[d]^{H^1(s_{N^+})} & H^1(\mathcal O_T(-4E))  \ar[d]^{H^1(s_{N'})} &  \\
            H^0(\O_W) \ar[r]^{\hspace{-1.7cm}\alpha_W} & H^1(\mathcal O_T(\overline{C}-2E+K_T)\otimes\I_{N^+}) \ar[r] & H^1(\mathcal O_T(\overline{C}-2E+K_T)\otimes\I_{N'}) \ar[r] & 0.} 
            \]
            We will show that one may choose $W \subset N^+=\{n_1,\ldots, n_{d(k-2)+1}  \}$
            appropriately so that

            \begin{equation} \label{eq:intcaz}
              \mathrm{Im}\, \alpha_W \cap \mathrm{Im}\, H^1(s_{N^+})=\{0\}.
            \end{equation}
            By \eqref{eq:o-2a2+}, this will imply the injectivity of $H^1(s_{N'})$ and thus condition \eqref{eq:o-2a2} for the nodes $N$ of a general curve $[\overline{C}]\in \widehat{V} \supset V'$. 

            We note that $h^1(\mathcal O_T(-4E))=3$, so that $\mathrm{Im}\, H^1(s_{N^+}) \cong \CC^3$.
By \eqref{eq:o-2a1+} we have the exact sequence
$$
\xymatrix{
  0 \ar[r] & H^0(\mathcal O_T(\overline{C}-2E+K_T)) \ar[r] &  H^0(\O_{N^+}) \ar[r]^{\hspace{-1.6cm} \alpha_{N^+}} &
H^1(\mathcal O_T(\overline{C}-2E+K_T)\otimes\I_{N^+});
}
$$
since $\overline{C}-2E+K_T\equiv (k-2)E+dR+K_T$ has no higher cohomology for $k\geq 3$, the image of $\alpha_{N^+}$ has dimension
\begin{eqnarray*}
  \deg N^+-\chi(\mathcal O_T(\overline{C}-2E+K_T)) & = & d(k-2)+1-d(k-3)= d+1 \\
                           & = & \rho(g,1,d)+g-d+3 \geq g-d+3.
\end{eqnarray*}
Since $H^0(\O_{N^+}) \cong \+_{i=1}^{d(k-2)+1} H^0(\O_{n_i})$, we may choose
a subscheme $W \subset N^+$ of degree $g-d$ so that \eqref{eq:intcaz} is satisfied.
\end{proof}

We can now prove the main theorem and its corollaries.

\begin{proof}[Proof of Theorem \ref{thm:main}]
  Let $[\varphi:C \to E] \in \H^{\tiny \mbox{prim}}_{g,k}(E)$ be general. First of all, we note that $\varphi$ does not factor as the composition $C\stackrel{\alpha}{\to}C'\stackrel{\varphi'}{\to}E$ of a cover $\alpha$ of degree $n\geq 2$ dividing $k$ onto a curve $C'$ of genus $h\geq 2$. This follows because, denoting by $\H_{g,n}(C')$ the Hurwitz scheme of genus $g$ and degree $n$ simply branched covers of a curve $C'$ as above, one computes that $$\dim \H^{\tiny \mbox{prim}}_{h,k/n}(E)+ \dim \H_{g,n}(C')=2h-2+2g-2-n(2h-2)<2g-2=\dim\H^{\tiny \mbox{prim}}_{g,k}(E).$$

 We now assume that $G^1_d(C)^{nc}$ is nonempty. Then we have a morphism $C \to T$,  birational onto its image $[\overline{C}]\in V^g(k,d)$. Let $\widetilde{\G}^1_d$ denote the family of pairs $\{\varphi:C \to E,\mathfrak{g}\}$ (up to isomorphism) such that
$[\varphi]  \in \H^{\tiny \mbox{prim}}_{g,k}(E)$ and $\mathfrak{g}\in G^1_d(C)^{nc}$. Proposition \ref{facile} then yields
\begin{eqnarray*}
  \dim \widetilde{\G}^1_d & \leq &\dim V^g(k,d)-\dim \Aut \PP^1 =  (2d+g-1)-3  \\
   & = & 2d+g-4=\rho(g,1,d)+2g-2. 
\end{eqnarray*}
The forgetful map
$\widetilde{\G}^1_d \to \H^{\tiny \mbox{prim}}_{g,k}(E)$ is dominant by assumption, and thus its fiber $G^1_d(C)^{nc}$ over a general $[\varphi]$ has dimension
 $\leq  \rho(g,1,d)$ and equality follows from Brill-Noether theory.
 
 It remains to show that $G^1_d(C)^{nc}$ is nonempty whenever $\rho(g,1,d) \geq 0$. 
 This is clearly true if $d < 2k$ and follows from Proposition \ref{prop:severi} when $d \geq 2k$.
 \end{proof}

\begin{proof}[Proof of Corollaries \ref{ye} and \ref{ye2}]
Take a general $[C]\in \M_{g,k}(E)$ so that there exists a pri\-mi\-tive degree $k$ cover $\varphi:C\to E$. Let $\mathfrak g$ be a $g^1_h$ on $C$ computing the gonality and assume $h<\frac{g+3}{2} $. By Theorem \ref{thm:main}, $\mathfrak g$ should be composed with $\varphi$ and thus $h=2k$. Let us now fix $d\geq \gon C$ and take a component $G$ of $G^1_d(C)$, which by standard Brill-Noether theory has dimension $\geq \rho(g,1,d)$. If $G$ is contained in the closure of $G^1_d(C)^{nc}$, then Theorem \ref{thm:main} yields $\dim G=\rho(g,1,d)$ and, if $k=2$, then $d\geq g-1$. Again Theorem \ref{thm:main} excludes the possibility that a general $\mathfrak g \in G$ has some base points $p_1,\ldots, p_s$ and $\mathfrak g(-p_1-\cdots-p_s)\in G^1_{d-s}(C)^{nc}$. It remains to treat the case where $G$ coincides with $G^1_{d,n}(C)$. When $k=2$ the dimension of $G^1_{d,n}(C)$ is independent on $n$, cf. \eqref{eq:star}, thus implying \eqref{dim2}. For $k>2$, the maximal dimension is instead obtained for $n=2$, cf. again \eqref{eq:star}, and this yields  \eqref{dim}. 

As concerns points (i),(ii) in Corollary \ref{ye} and (i),(ii),(iii) in Corollary \ref{ye2}, we stress that $G^1_d(C)^{nc}$ is union of irreducible components of $G^1_d(C)$ as soon as it is nonempty. Furthermore, since  $G^1_d(C)^{nc}$ has pure dimension $\rho(g,1,d)$, the locus $G^1_{d,n}(C)$ is not contained in $G^1_d(C)$ as soon as its dimension computed in \eqref{eq:star} is $\geq \rho(g,1,d)$. If $G^1_{d,n}(C)$ were contained in some other $G^1_{d,m}(C)$ with $n\neq m$, by dimensional reasons one should have $n>m$; but then a general linear series in $G^1_{d,m}(C)$ would have more base points than a general element of $G^1_{d,n}(C)$, thus contradicting the containment. As a consequence, the locus $G^1_{d,n}(C)$ is an irreducible component of $G^1_d(C)$ as soon as its dimension is $\geq \rho(g,1,d)$.
\end{proof}

\end{document}